\documentclass{amsart}

\usepackage{mathptmx, amscd, amssymb, amsmath, enumerate, slashbox, hyperref, float}

\usepackage[utf8]{inputenc}
\usepackage[all]{xy}
\usepackage{color}
\usepackage{hyperref}

\usepackage{amsmath, mathtools}
\usepackage{hyperref}
\hypersetup{colorlinks=true}
\usepackage{extpfeil}
\usepackage{amssymb}
\usepackage{amsfonts}
\usepackage{graphicx}
\usepackage{mathrsfs}
\usepackage{stmaryrd}
\usepackage{stmaryrd}
\usepackage{mathtools}
\usepackage{amscd}
\usepackage[all]{xy}
\usepackage{hyperref}
\usepackage{amsthm, enumitem }
\usepackage{cleveref}
\usepackage{verbatim}
\usepackage{amscd}
\usepackage{mathtools}
\usepackage{comment}
\usepackage{diagbox,bm}

\usepackage{tikz}
\usetikzlibrary{matrix,arrows,decorations.pathmorphing,cd}

\newtheorem{thm}{\bf Theorem}[section]

\newtheorem*{thma}{Theorem A}
\newtheorem*{thmb}{Theorem B}
\newtheorem*{thmc}{Theorem C}
\newtheorem*{thmd}{Theorem D}

\newtheorem{prop}[thm]{\bf Proposition}
\newtheorem{lemma}[thm]{\bf Lemma}
\newtheorem{cor}[thm]{\bf Corollary}
\newtheorem{conjecture}[thm]{\bf Conjecture}
\theoremstyle{definition}
\newtheorem{definition}[thm]{\bf Definition}
\theoremstyle{remark}
\newtheorem{remark}[thm]{\bf Remark}

\newtheorem{question}[thm]{\bf Question}
\newtheorem{example}[thm]{\bf Example}
\numberwithin{equation}{section}

\newcommand{\Ass}{\operatorname{Ass}}

\newcommand{\id}{\operatorname{id}}
\newcommand{\Ext}{\operatorname{Ext}}

\newcommand{\Tor}{\operatorname{Tor}}
\newcommand{\Hom}{\operatorname{Hom}}

\newcommand{\codim}{\operatorname{codim}}
\newcommand{\depth}{\operatorname{depth}}

\newcommand{\coker}{\operatorname{coker}}

\newcommand{\Den}{\operatorname{Den}}

\newcommand{\rank}{\operatorname{rank}}

\newcommand{\p}{\mathfrak{p}}
\newcommand{\m}{\mathfrak{m}}
\newcommand{\DenRt}{\textnormal{Den}^R(t)}
\newcommand{\DenR}{\textnormal{Den}^R}

\newcommand{\fm}{\mathfrak{m}}
\newcommand{\ls}{\leqslant}
\newcommand{\gs}{\geqslant}
\newcommand{\fp}{\mathfrak{p}}
\newcommand{\fn}{\mathfrak{n}}
\newcommand{\uind}{\mathfrak{u}}
\newcommand{\unind}{\mathfrak{u}}

\DeclareMathOperator{\cx}{cx}
\DeclareMathOperator{\px}{px}

\DeclareMathOperator{\lr}{lr}
\newcommand{\w}{\omega}
\newcommand{\elll}{\ell\ell}
\newcommand{\NN}{\mathbb{N}}

\def\fn{\mathfrak{n}}

\def\fp{\mathfrak{p}}

\def \RR{\mathbb R}

\def \NN{\mathbb N}
\def \ZZ{\mathbb Z}

\def \H{\mathcal H}

\dedicatory{{Dedicated to Professor Bernd Ulrich on the occasion of his sixty-fifth birthday.}}

\begin{document}
\title[Extremal growth of Betti numbers and
trivial vanishing of (co)homology]{Extremal growth of Betti numbers and
trivial vanishing of (co)homology}

\author[Justin Lyle]{Justin Lyle}
\address{Justin Lyle\\ Department of Mathematics \\ University of Kansas\\405 Snow Hall, 1460 Jayhawk Blvd.\\ Lawrence, KS 66045}
\email{justin.lyle@ku.edu }

\author[Jonathan Monta\~no]{Jonathan Monta\~no}
\address{Jonathan Monta\~no \\ Department of Mathematical Sciences  \\ New Mexico State University  \\PO Box 30001\\Las Cruces, NM 88003-8001}
\email{jmon@nmsu.edu}

\begin{abstract}
A  Cohen-Macaulay local ring $R$ satisfies  trivial vanishing if  $\Tor_i^R(M,N)=0$ for all large $i$ implies $M$ or $N$ has finite projective dimension. If $R$ satisfies trivial vanishing then we also have that $\Ext^i_R(M,N)=0$  for all large $i$ implies  $M$ has finite projective dimension or $N$ has finite injective dimension. In this paper, we establish obstructions for the failure of trivial vanishing in terms of the asymptotic growth of the Betti and Bass numbers of the modules involved. These, together with results of Gasharov and Peeva,  provide sufficient conditions for $R$ to satisfy trivial vanishing; we provide sharpened conditions when $R$ is generalized Golod. Our methods allow us to settle the Auslander-Reiten conjecture in several new cases. In the last part of the paper, we provide criteria for the Gorenstein property based on consecutive vanishing of Ext. The latter results improve similar statements due to Ulrich, Hanes-Huneke, and Jorgensen-Leuschke.
  \end{abstract}

\keywords{Cohen-Macaulay rings, Ext, Tor, Betti numbers, resolutions, generalized Golod rings, complexity, Auslander-Reiten conjecture, Gorenstein rings.}
\subjclass[2010]{13D07, 13D02, 13C14, 13H10, 13D40.}

\maketitle

\section{Introduction}

Let $(R,\m,k)$ be a Cohen-Macaulay (CM) local ring.  In this paper we consider the following conditions on the vanishing of (co)homology.

\begin{enumerate}

\item For any finitely generated $R$-modules $M$ and $N$, $\Tor_i^R(M,N)=0$ for $i \gg 0$ implies that either $M$ or $N$ has finite projective dimension.

\item For any finitely generated $R$-modules $M$ and $N$, $\Ext^i_R(M,N)=0$ for $i \gg 0$ implies that either $M$ has finite projective dimension or $N$ has finite injective dimension.
\end{enumerate}

Condition (1) is the strongest possible condition one can impose on the asymptotic vanishing of Tor, while condition (2) is the strongest one can impose on the asymptotic vanishing of Ext. 
While (1) always implies (2), these conditions  are equivalent under mild assumptions, e.g., if $R$ has a canonical module (Theorem \ref{equiv}).  Following Jorgensen and \c{S}ega (\cite{JS}), we say $R$ satisfies {\it trivial vanishing} if (1), and thus (2), holds for $R$.

In the past few decades, these rigidity conditions have gained much attention, in particular in connection with the following long-standing conjecture.

\begin{conjecture}[{Auslander-Reiten \cite{AR}}]\label{ARC}
Let $R$ be a Noetherian ring and $M$ a finitely generated $R$-module. If $\Ext^i_R(M,M\oplus R)=0$ for every $i>0$, then $M$ is projective.
\end{conjecture}

Recently, much research activity has been centered on proving this conjecture.  However, up to date, it remains open (see for example \cite{HL},\cite{HS04},\cite{CIST},\cite{NS},\cite{Lindo}, and see \cite[Appendix A]{CH} for a survey on the topic).

The study of trivial vanishing was pioneered by Huneke and Wiegand (\cite{HW}) and independently by Miller (\cite{CMiller}), who show this condition holds for hypersurface rings.  This was later extended by Jorgensen to show that any Golod ring satisfies trivial vanishing (\cite{Jorgensen}).  On the other hand, \c{S}ega showed that  trivial vanishing fails if $\codim(R)\gs 2$ and the completion $\hat{R}$ has {embedded deformations} (i.e., $\hat{R}\cong Q/(\underline{f})$  for some local ring $(Q,\fp)$  and $Q$-regular sequence $\underline{f}\subseteq \fp^2$) (\cite{Sega}). In particular, any complete intersection of codimension larger than one cannot satisfy trivial vanishing, a result originally obtained by Avramov and Buchweitz using the theory of support varieties (\cite{AvrBuch}).

To summarize these results, we have a good understanding of the rigid behavior of $\Ext$ and $\Tor$ over rings whose modules have Betti numbers of extremal growth; modules over Golod rings have the fastest growth of Betti numbers while those over complete intersections have the slowest (\cite[5.3.2 and 8.1.2]{Avr}).  A  unification of these settings  is given by the {\it generalized Golod rings} of which both are examples; see \cite{Avr94} or Subsection \ref{gGsection} for the definition. While the behavior of Betti numbers over generalized Golod rings is more subtle, they still possess a wealth of structure.  In particular, any module $M$ over a generalized Golod ring has a rational Poincar\'e series and these rational expressions can be made to share a common denominator (\cite[1.5]{Avr94}). By studying growth rates of Betti numbers in general, we establish a sufficient numerical condition for any CM local ring to satisfy trivial vanishing, and we provide a refined version when $R$ is generalized Golod.  Using this result, we are able to establish the Auslander-Reiten conjecture in a number of new cases.

In order to describe our main results, we need to introduce some preliminary notation. We denote by $e_R(M)$ (or simply $e(R)$ when $M=R$) the {\it (Hilbert-Samuel) multiplicity} of the $R$-module $M$, and we write $\mu(M)$ for the minimal number of generators of $M$. We let $\codim R:=\mu(\m)-\dim R$ denote the {\it codimension} of $R$. In our proofs, we may assume $R$ has an infinite residue field (See Proposition \ref{finite flat} and Remark \ref{reductions}). Then we can define the {\it Loewy length} $\elll(R)$ as the maximum among $\min\{i\mid \fm^i\subseteq J\}$ where $J$ ranges over the minimal reductions of $\fm$ (see Definition \ref{Low}).  We now present our first theorem; see Theorems \ref{main} and \ref{maingengolod}.

\begin{thma}
Let $R$ be a CM local ring and set $c =  \codim(R)$ and $\ell=\elll(R)$.  Assume one of the following conditions holds.
\begin{enumerate}
\item[$(1)$] $e(R)<\frac{4c+2\ell-1-\sqrt{8c+4\ell-3}}{2}$.
\item[$(2)$] $R$ is generalized Golod and $e(R)\ls2c+\ell-4$.
\end{enumerate}
Then  $R$ satisfies trivial vanishing.
\end{thma}

The authors in \cite{JS} provide examples of rings $R$ that do not satisfy trivial vanishing for which the completion has no embedded deformations. The first example is a Gorenstein ring with $e(R)=12$ and $\codim(R)=5$. From \cite{Sega}, we know this example is minimal with respect to codimension, as no such Gorenstein ring exists with $\codim(R)\ls 4$. The second example from  \cite{JS}  has $e(R)=8$ and $\codim(R)=4$.  In our next theorem, we show that both examples are minimal with respect to codimension and multiplicity; see Propositions \ref{lowdim} and \ref{codim3}.

\begin{thmb}\label{thmB}
Let $R$ be a CM local ring with $\codim(R)\neq 1$ and assume it satisfies one of the following conditions.
\begin{enumerate}
\item[$(1)$] $\codim(R)\ls 3$.
\item[$(2)$] $e(R)\ls 7$.
\item[$(3)$]  $e(R)\ls 11$ and $R$ is Gorenstein.
\end{enumerate}
Then $R$  satisfies trivial vanishing if and only if the completion $\hat{R}$ has no embedded deformations.
\end{thmb}

We note that the assumption $\codim(R)\neq 1$ is necessary by \cite{HW}.  We also remark that a key point to Theorem B is that CM rings of small codimension and multiplicity tend to be generalized Golod (see Example \ref{gGolod}). The only case covered by Theorem B where we do not know $R$ is generalized Golod is the case where Artinian reductions of $R$ have $h$-vector $(1,4,2)$.

As a consequence of our work, we are able to verify the Auslander-Reiten conjecture in some new cases; see Theorem \ref{mainAR} and Corollary \ref{e8}.

\begin{thmc}
Let $R$ be a CM local ring. Assume $R$ satisfies one of the following conditions.
\begin{enumerate}
\item[$(1)$] $e(R)\ls \frac{7}{4}\codim(R)+1$.
\item[$(2)$]   $e(R)\ls  \codim(R)+6$ and $R$ is Gorenstein.
\end{enumerate}
Then the Auslander-Reiten conjecture $($Conjecture \ref{ARC}$)$ holds for $R$. In particular, the conjecture holds if $e(R)\ls 8$, or if $e(R)\ls 11$ and $R$ is Gorenstein.
 \end{thmc}
 In several other cases we show that a stronger condition is satisfied, namely,  the {\it uniform Auslander condition} which  implies Conjecture \ref{ARC} by \cite{CH}; see Corollary \ref{UAC}. 

In \cite{Ulrich}, Ulrich provides conditions on a finitely generated $R$-module $M$ so that the vanishing of $\Ext^i_R(M,R)=0$ for $1 \ls i \ls \dim(R)$ forces $R$ to be Gorenstein. In this sense, $M$ can be used a test module for the Gorenstein property. In \cite{JL07} and \cite{HH}, some variations on this result are included.  In the last part of the paper, we expand upon these results by proving the following; see Theorem \ref{stronghh}. We recall that the ring $R$ is {\it generically Gorenstein} if $R_\fp$ is Gorenstein for every $\fp\in \Ass(R)$.

\begin{thmd}
Let $R$ be a generically Gorenstein CM local ring that has a canonical module.  Assume there exists a Maximal Cohen-Macaulay $R$-module $M$  with  $e_R(M)\ls 2\mu(M)$ and such that $\Ext^i_R(M,R)=0$ for $1 \ls i \ls \dim(R)+1$. Then $R$ is Gorenstein.
\end{thmd}

We now describe the structure of the paper. In Section \ref{notation}, we set the notation that is used throughout the paper. We also discuss preliminary results and definitions that are necessary in our proofs. In Section \ref{tvS}, we give the definition of trivial vanishing and consider its behavior under change of rings, specifically under local maps of finite flat dimension and hyperplane sections. In Section \ref{mainResults}, we include our main results; here we prove Theorem A ( Theorems \ref{main} and \ref{maingengolod}) and  Theorem B (Propositions \ref{lowdim} and \ref{codim3}). Section \ref{ARSection} includes consequences for the uniform Auslander condition and the Auslander-Reiten conjecture. In particular, this section includes the proof of Theorem C (Theorem \ref{mainAR}). The last section, Section \ref{Gorenst}, includes the proof of Theorem D (Theorem \ref{stronghh}) and related results.

\section{Notation and preliminary results}\label{notation}

Throughout this paper we  assume all rings are Noetherian and all modules are finitely generated. Let $(R,\fm,k)$ be a local ring of dimension $d$ and  $M$ be an $R$-module.  We denote  by $\beta_i^R(M)=\dim_k \Tor^R_i(k,M)$ the  {\it $i$th Betti number} of $M$, $\mu(M)=\beta_0^R(M)$, the minimal number of generators of $M$, and $P_M^R(t)=\sum_{i=0}^{\infty} \beta_i^R(M)t^n$ the  {\it Poincar\'e series} of $M$ over $R$. We also write $\Omega_i^R(M)$ for the {\it $i$th syzygy} of  $M$ and   $r(R)=\dim_k \Ext_R^{\depth R}(k,M)$ for its {\it type}.   We write $\hat{R}$ for the $\fm$-adic completion of $R$.

If $M$ has dimension $s$, we denote by $$e_R(M)=\lim_{n\rightarrow \infty} \frac{s!\lambda(M/\fm^n)}{n^s}$$ the {\it Hilbert-Samuel multiplicity} of $M$, where  $\lambda(N)$ denotes the length of the $R$-module $N$. If $M=R$, we simply write  $e(R)$. The following invariant has appeared in different forms and has played an important role in several results in the literature (see for example \cite{HH},\cite{HS04}, \cite{JL07},\cite{Ulrich}). We provide the following notation and name in order to simplify some of the statements.

\begin{definition} \label{uind}
Let $R$ be a local ring and $M$  a non-zero $R$-module.  We define the {\it Ulrich index} of $M$, denoted by $\uind_R(M)$, as
the ratio $$\uind_R(M):=\dfrac{e_R(M)}{\mu(M)}.$$
\end{definition}

\begin{remark}
We note that when $M$ is CM, we always have the inequality $e_R(M)\gs \mu(M)$, and therefore $\unind_R(M)\gs 1$. The Maximal Cohen-Macaulay (MCM) $R$-modules such that $\unind_R(M)=1$ are the so-called {\it Ulrich modules}, therefore the Ulrich index provides a measure of how far a module is from being Ulrich.
\end{remark}

The next lemma  is the content of \cite[2.1]{JL07} stated in our terminology (see also \cite[2.1]{JL08}). We remark that, although the original version of part (1) assumes $N$ is  MCM, its proof does not actually requires this condition. If $R$ has a canonical module $\omega_R$, we write  $$M^{\vee}:=\Hom_R(M,\omega_R)$$ for its canonical dual. In the following statement we use the convention $\beta_{-1}^R(N)=0$.

\begin{lemma}[{\cite[2.1]{JL07}, \cite[2.1]{JL08}}]\label{superThm}
Let $R$ be a CM local ring of dimension $d$. Let $M$ and $N$ be $R$-modules and assume $M$ is  CM. Fix $n\in \NN$  and assume one of the following conditions holds
\begin{enumerate}
\item[$(1)$] $n\gs \dim M$ and  $\Tor_i^R(M,N)=0$  for every $n-\dim M\ls i\ls n$, or
\item[$(2)$] $n\gs 0$,  $R$ has a canonical module $\omega_R$, $N$ is MCM, and  $\Ext^i_R\big(M,N^{\vee}\big)=0$ for  every $n+ d-\dim M\ls i\ls n+d$.
\end{enumerate}
Then $\beta_n^R(N)\ls (\unind_R(M)-1)\beta_{n-1}^R(N)$.
\end{lemma}

The {\it codimension} of $(R,\fm)$ is defined as $\codim(R)=\mu(\fm)-\dim(R)$.  We note that, if $R$ is Artinian, then $\codim(R)$ is simply $\mu(\fm)$. In the following remark we explain the process of reduction to the Artinian case.

\begin{remark}\label{reductions}
Let $R$ be a CM local ring, it is always possible to reduce $R$ to an Artinian ring  with the same codimension and multiplicity in the following way. First, we  extend $R$ to the faithfully flat $R$-algebra $R[X]_{\fm R[X]}$ to assume $R$ has infinite residue field. Then, we  mod out $R$ by a minimal reduction of $\fm$ .
\end{remark}

 We recall that the {\it Loewy length} of an Artinian local ring $R$ is defined as $\elll(R)=\min \{i\mid \fm^i=0\}$. We extend this definition to arbitrary CM local rings as follows.

\begin{definition}\label{Low}
Let $R$ be a CM local ring. We define the {\it Loewy length} of $R$, denoted by $\elll(R)$, as  the maximum of the  Lowey lengths of the Artinian reductions of $R$ as in Remark \ref{reductions}. We remark that when $R$ is equicharacteristic, $\elll(R)$ is achieved when modding out by a general reduction of $\m$ \cite[5.3.3]{Fo06}.
\end{definition}

The following result of Gasharov and Peeva will be crucial to the proofs of the main results of this paper. While their result is stated for Artinian rings,  we include here a version of it for arbitrary CM  local rings. 
We remark that although part (1) is not explicitly stated therein, it is included in the proof of  \cite[2.2]{GP90}.

\begin{prop}[{\cite[2.2]{GP90}, \cite[Proposition 3]{Peeva}}]\label{eventInc}
Let  $R$ be a CM local ring of dimension $d$ and  $M$ be  an $R$-module. Set $c =  \codim(R)$, $e=e(R)$, and $\ell=\elll(R)$. Then for every $n>  \max\{d-\depth M, \mu(M)\}$ we have,
\begin{enumerate}
\item[$(1)$] $\beta_{n}^R(M)\gs c\beta_{n-1}^R(M)-(e-c-\ell+2)\beta_{n-2}^R(M)$,
\item[$(2)$] $\beta_{n}^R(M)\gs (2c-e+\ell-2)\beta_{n-1}^R(M)$, and
\item[$(3)$] If $c\gs 3$ and $2c  -e+\ell - 2=1$,  there either the Betti numbers of $M$ are eventually constant, or there exists  $C>1$ such that $\beta_{n}^R(M)\gs C\beta_{n-1}^R(M)$ for every $n\gg 0$.
\end{enumerate}
\end{prop}

We now discuss a class of rings introduced by Avramov in \cite{Avr94}. 

\subsection{Generalized Golod rings.}\label{gGsection}  An {\it acyclic closure} of $R$ is a DG-algebra resolution  of $k$ constructed via Tate's process of adjoining variables  to kill cycles \cite{Tate}.   
The process starts with the Koszul complex $K^R=R\langle X_1\rangle$, and, for $n\gs 1$,  it inductively adjoins variables $X_{n+1}$ in homological degree $n+1$ in such a way that the classes of $\partial(X_{n+1})$ minimally generate the homology $H(R\langle X_{\ls n} \rangle)$. Here,  $X_n$ are exterior variables if $n$ is odd and  divided powers variables if $n$ is even.
 Setting $X=\bigcup_{n\gs 1} X_n$, the resulting acyclic closure $R\langle X\rangle$ is a minimal free resolution of $k$ (\cite{Gu, Scho}). We refer the reader to \cite[Section 6.3]{Avr} for more information.

In \cite{Avr94}, Avramov defines a local ring $R$ to be {\it generalized Golod $($of level $\ls n)$} if the DG-algebra $R\langle X_{\ls n}\rangle$ admits a {\it trivial Massey operation} for some $n\gs 1$ (see \cite[5.2.1]{Avr}). We notice that the classical Golod rings are precisely  the generalized Golod rings of level $\ls 1$. One of the main motivations for introducing this class of rings is the following theorem.

\begin{thm}[{\cite[1.5]{Avr}}]\label{AvrDen}
Let $R$ be a generalized Golod ring, then there exists  a polynomial $\DenRt\in \ZZ[t]$ such that for every $R$-module $M$ there exists  $p_M(t)\in \ZZ[t]$ giving $$ P_M^R(t)=\frac{p_M(t)}{\DenRt}.$$
 Moreover, when $M=k$,  all the roots of $p_k(t)$ have  magnitude one.
\end{thm}

We now present some classes of rings that are generalized Golod.

\begin{example}\label{gGolod} The local ring $R$ is generalized Golod in any of the following situations:
\begin{enumerate}
\item $R$ is a complete intersection (\cite{Tate}).
\item $R$ is Golod (\cite{GL}).
\item $\mu(\fm)-\depth(R)\ls 3$ (\cite[6.4]{AKM}), see also  \cite[3.5]{Avr89}.
\item $R$ is Gorenstein and $\mu(\fm)-\depth(R)\ls 4$  (\cite[6.4]{AKM}), see also \cite[3.5]{Avr89}.
\item $R$ is one link from a complete intersection (\cite[6.4]{AKM}).
\item $R$ is two links from a complete intersection  and $R$ is Gorenstein (\cite[6.4]{AKM}).
\item $R$ is almost a complete intersection of codimension four, and 2 is a unit in $R$ (\cite[4.2]{KP}).
\item $R$ is presented by Huneke-Ulrich ideals of full codimension (\cite[5.2]{K}).
\item $R$ is Gorenstein and $e(R)\ls 11$ (\cite[6.9]{Gupta}).
\item $R$ is Gorenstein, $\fm^4=0$, and $\mu(\fm^2)\ls 4$ (\cite[6.9]{Gupta}).
\item $R$ is presented by certain determinantal ideals of full codimension (\cite[6.5]{AKM}).
\item $R$ is a CM {\it stretched} ring, or an {\it almost stretched} Gorenstein ring (\cite[5.4]{CDGetc}, \cite[6.1]{Gupta}).
\item Certain {\it compressed} Artinian rings  (\cite[5.1]{RS}, \cite[7.1]{KSV}).
\end{enumerate}
\end{example}

The following invariant has been studied by other authors as it  measures the growth of Betti numbers of  modules (\cite[Section 4]{Avr}).  We  give this invariant the following name and notation for  clarity of our exposition.

\begin{definition}
Let  $M$ be an $R$-module of infinite projective dimension. We define the {\it limit ratio} of $M$, denoted by $\lr_R(M)$, as the formula $$\lr_R(M):=\limsup_{n\to \infty}\frac{\beta_{n+1}^R(M)}{\beta_{n}^R(M)}.$$
\end{definition}
Clearly $\lr_R(M)\gs 1$. We remark that while it is unknown if $\lr_R(M)$ is always finite (cf. \cite[4.3.1]{Avr}), this is indeed the case when $R$ is a CM local ring (\cite[4.2.6]{Avr}). The limit ratio is naturally related to the {\it curvature} of $M$ 
 (see paragraph before \cite[4.3.6]{Avr}), 
 and to the   {\it complexity} of  $M$,  i.e,  
 $$\cx_R(M)=\inf\{t\in\NN\mid \text { there exists }\beta\in\RR\text{ such that }\beta_n^R(M)\ls \beta n^{t-1} \text{ for every } n\gs 1\}.$$
Note that, by definition, $\cx_R(M)=0$ if and only if $M$ has finite projective dimension. 

In fact, in  Problems 4.3.6 and 4.3.9 of  \cite{Avr}, Avramov  asks whether the limit in the definition of $\lr_R(M)$ always exists. 
By a result of Sun, this limit exists for  modules  of infinite complexity over a generalized Golod ring.

 \begin{prop}[{\cite[Corollary]{Sun98}}]\label{ratLimExists}
Let $R$ be a generalized Golod ring and $M$ be  an $R$-module. 
Assume $\cx_R(M)=\infty$. 
Then the limit $\displaystyle\lim_{n\to \infty}\frac{\beta_{n+1}^R(M)}{\beta_{n}^R(M)}$ exists, and is greater than $1$.
 \end{prop}

\subsection{MCM approximations.} 

The following result of Auslander and Buchweitz, and the subsequent remark, allow us to often replace arbitrary finitely generated modules for MCM modules when dealing with vanishing of Ext. 

\begin{thm}[{\cite[Theorem A]{AB}}]\label{MCMaprox}
Let $R$ be a CM local ring with canonical module $\omega_R$ and let $N$ be an $R$-module. Then there exist $R$-modules $Y$ and $L$, such that  $L$ is MCM, $Y$ has finite injective dimension, and  they fit in a short exact sequence $0 \to Y \to L \to N \to 0$.
\end{thm} 

The exact sequence from the previous lemma is commonly referred as an {\it MCM  approximation} of $N$. 

\begin{remark} \label{repl}
We note that, since $Y$ in Theorem \ref{MCMaprox} has finite injective dimension, for every $R$-module $M$ we have $\Ext^i_R(M,N)=0$ for $i\gg 0$ if and only if $\Ext^i_R(M,L)=0$ for $i\gg 0$. In particular, if $M=k$ , then $N$ has finite injective dimension if and only if $L$ does. 
\end{remark}

We now recall another  notion of complexity. The {\it plexity}  of $N$ is defined in terms of its {\it Bass numbers}, i.e., $\mu_R^i(N)=\dim_k \Ext_R^i(k,N)$. We have, 
$$\px_R(N)=\inf\{t\in\NN\mid \text { there exists }\mu\in\RR\text{ such that }\mu^n_R(N)\ls \mu n^{t-1} \text{ for every } n\gs 1\}.$$
For more information about these notions of complexity, see \cite[Section 4.2]{Avr} and \cite{Avr96}. There is a direct relation between the two complexities, we discuss this in the following remark.

\begin{remark}\label{sameCx}
Let $R$ be a CM local ring of dimension $d$ and with a canonical module $\omega_R$. Let $0 \to Y \to L \to N \to 0$ be an MCM approximation of $N$  (cf. Theorem \ref{MCMaprox}), then 
$$\mu^{i+d}_R(N) = \beta_{i}^R((L)^\vee) \text{ for }i> 0,\qquad \text{and then},\qquad \px_R(N)=\cx_R((L)^\vee).$$ 
To see this, choose $\underline{x}=x_1,\ldots, x_d$ a maximal regular sequence on $R$, $L$, and $(L)^\vee$. The claim now follows by observing that 
\begin{align*}\label{relateBass}
\mu^{i+d}_R(N) =\mu^{i+d}_R(L) &=\mu^{i}_{R/\underline{x}R}(L/\underline{x}L)\\
&=\beta_i^{R/\underline{x}R}\big( (L/\underline{x}L) ^\vee\big)=\beta_i^{R/\underline{x}R}\big( (L)^\vee/\underline{x}(L)^\vee\big)=\beta_i^R((L)^\vee)
\end{align*}
for every $i> 0$, where the third equality  holds by Matlis duality. 
\end{remark}

\section{Trivial vanishing}\label{tvS}

In this section we present the definition of the trivial vanishing condition for CM local rings and prove some preliminary results.

\begin{definition}\label{tv}
A CM local  ring  $R$ satisfies {\it trivial Tor-vanishing} if, for any $R$-modules  $M$ and $N$,  $\Tor_i^R(M,N)=0$ for $i\gg 0$ implies  $M$ or $N$ has finite projective dimension. We say $R$ satisfies  {\it trivial Ext-vanishing}  if   $\Ext^i_R(M,N)=0$ for $i\gg 0$ implies $M$ has finite projective dimension or $N$ has finite injective dimension. If both conditions are satisfied, we simply say $R$ satisfies {\it trivial vanishing}. 
\end{definition}

The relation between the two trivial vanishing conditions is explained by the following result.

\begin{thm}\label{equiv}
Let $R$ be a CM local  ring. 
\begin{enumerate}
\item[$(1)$] If  $R$ satisfies trivial Tor-vanishing, then $R$ satisfies trivial vanishing.
\item[$(2)$] If $R$ has a canonical module, the two trivial vanishing conditions $($Ext and Tor$)$ are equivalent.  
\end{enumerate}
\end{thm}

For the proof of this theorem we need some preparatory results.


\begin{lemma}\label{dualexttor}
Let $R$ be a CM local ring with canonical module $\omega_R$. Let $M$ and $N$ be $R$-modules and assume  that $N$, $M \otimes_R N^{\vee}$,  and  $\Omega_1^R(M) \otimes_R N^{\vee}$ are MCM.  Then,  $\Ext^1_R(M,N)=0$ if and only if $\Tor^R_1(M,N^{\vee})=0$. 
\end{lemma}

\begin{proof}
From the exact sequence 
$0 \rightarrow \Omega^1_R(M) \rightarrow R^{\mu(M)} \xrightarrow{p} M \rightarrow 0,$
and the natural isomorphism  $(-\otimes_R N^\vee)^\vee\cong \Hom_R(-,N)$, 
we obtain $$(\ker(p \otimes \id_{N^{\vee}}))^\vee\cong \coker(\Hom_R(p,N)).$$
If $\Ext^1_R(M,N)=0$, then $\coker(\Hom_R(p,N))=\Hom_R(\Omega^1_R(M),N)$, and so $\ker(p \otimes_R \id_{N^{\vee}})=\Omega^1_R(M) \otimes_R N^{\vee}$.  Thus $\Tor_1^R(M,N^{\vee})=0$.  
Similarly, if $\Tor^R_1(M,N^{\vee})=0$, then $\ker(p \otimes_R \id_{N^{\vee}})=\Omega^1_R(N) \otimes_R N^{\vee}$ and so $\coker(\Hom_R(p,N))=\Hom_R(\Omega^1_R(M),N)$.  Thus $\Ext^1_R(M,N)=0$, and the proof is complete.
\end{proof}

\begin{lemma}\label{exp}
Let $R$ be a CM local ring of dimension $d$ and with canonical module $\omega_R$. Let $M$ and $N$ be $R$-modules and assume $N$ is MCM. Then for each $n\in \NN$ the following  hold.

\begin{enumerate}
\item[$(1)$] If  $\Ext^i_R(M,N)=0$ for $1\ls i\ls d$, then $M\otimes_R N^\vee$ is MCM.

\item[$(2)$] If $\Ext^i_R(M,N)=0$ for $1 \ls i \ls d+n$, then $\Tor_i^R(M,N^{\vee})=0$ for $1 \ls i \ls n$.

\item[$(3)$]  If $\Tor_i^R(M,N^{\vee})=0$ for $1 \ls i \ls d+n$, then $\Ext^i_R(M,N)=0$ for $d+1 \ls i \ls d+n$.  
\end{enumerate}
\end{lemma}

\begin{proof}
The proof of (1) is included in \cite[5.3]{DE17}.  We remark that, although $M$ is assumed to be MCM in \cite{DE17}, this hypothesis is not needed in the proof.

By (1), the assumption of $(2)$ implies $\Omega_i^R(M) \otimes_R N^{\vee}$ is MCM for  $0 \ls i \ls n$. Hence, (2) follows from applying Lemma \ref{dualexttor}.

For $(3)$, since $N^{\vee}$ is MCM, the depth lemma implies $\Omega_i^R (M) \otimes N^{\vee}$ is MCM for $d\ls i\ls d+n$. Then the conclusion follows from Lemma \ref{dualexttor}.
\end{proof}

We now consider how the trivial vanishing conditions behave under extensions of finite flat dimension and hyperplane sections. Many of our arguments draw inspiration from Section 2 of \cite{CH2}. For our purposes we only consider the case where $R$ is CM.  For similar results without this hypothesis, see \cite{AINS2}.

\begin{prop}\label{finite flat}
Let  $(R,\m) \to (S,\mathfrak{n})$ be  a local ring map with $R$ CM.

\begin{enumerate}
\item[$(1)$] Assume $S$ has finite flat dimension over $R$. If $S$ satisfies trivial Tor-vanishing, then so does $R$.
\item[$(2)$]  Assume $S$ is flat over $R$. If $S$ satisfies trivial $\Ext$-vanishing, then so does $R$.
\end{enumerate}
\end{prop}

\begin{proof}

We begin with (1). Assume  $\Tor_i^R(M,N)=0$ for $i \gg 0$.  By replacing  $M$ and $N$ with sufficiently high syzygies , we may assume $M$, $N$, and $M \otimes_R N$ are MCM, and that $\Tor_i^R(M,N)=0$ for $i>0$.  Let $F^M_{\bullet}$ and $F_{\bullet}^N$ be minimal free resolutions of $M$ and $N$, respectively. Then $F^M_{\bullet} \otimes_R F^N_{\bullet}$ is a minimal free resolution of $M \otimes_R N$.  
By \cite[3.4(2)]{CFF}, we have $F_{\bullet}^M \otimes_R S$ and $F_{\bullet}^N \otimes_R S$ are minimal free $S$-resolutions of $M \otimes_R S$ and $N \otimes_R S$, respectively. Therefore,
$$\Tor^S_i(M \otimes_R S,N \otimes_R S)=H_i(F_{\bullet}^M \otimes_R S) \otimes_S (F_{\bullet}^N \otimes_R S)\cong H_i((F_{\bullet}^M \otimes_R F_{\bullet}^N) \otimes_R S)=0,$$
where the last equality follows from \cite[3.4(2)]{CFF}.  As $S$ satisfies trivial vanishing,  we have that either  $M \otimes_R S$ or $N \otimes_R S$ has finite projective dimension over $S$.  Since $F_{\bullet}^M \otimes_R S$ and $F_{\bullet}^N \otimes_R S$ are minimal free resolutions of these modules, it follows that either $M$ or $N$ has finite projective dimension over $R$, completing the proof.

We now prove (2). Assume $S$ satisfies trivial $\Ext$-vanishing and assume $M$ and $N$ are $R$-modules with $\Ext^i_R(M,N)=0$ for $i \gg 0$.  Since $\Ext^i_R(M,N) \otimes_S S \cong \Ext^i_S(M \otimes_R S,N \otimes_R S)$ for each $i$,  we conclude $\Ext^i_S(M \otimes_R S,N \otimes_R S)=0$ for $i \gg 0$.  By assumption we must have that $M \otimes_R S$ has finite projective dimension over $S$ or $N \otimes_R S$ has finite injective dimension over $S$. Therefore, the same holds over $R$  (\cite[Corollary 1]{FT}). This concludes the proof.
\end{proof}

The following proposition is crucial for our results as it allows us to pass to complete rings in the proofs, or even Artinian, and with infinite residue field. 

\begin{prop}\label{pass}
Let $R$ be a CM local ring. The  following conditions are equivalent

\begin{enumerate}

\item[$(1)$] $R$ satisfies trivial $\Tor$-vanishing.

\item[$(2)$] $R/\underline{x}R$ satisfies trivial $\Tor$-vanishing for an $R$-regular sequence $\underline{x}\in \m\setminus \m^2$ .

\item[$(3)$] $\hat{R}$ satisfies trivial $\Tor$-vanishing.

\end{enumerate}

\end{prop}

\begin{proof}
Each condition $(2)$-$(3)$ implies $(1)$ by Proposition \ref{finite flat} (1).  We show $(1)$ implies both $(2)$ and $(3)$, starting with $(1)$ implies $(2)$. 

Assume $R$ satisfies trivial $\Tor$-vanishing and let $M$ and $N$ be $R/\underline{x}R$-modules such that $\Tor^{R/\underline{x}R}_i(M,N)=0$ for $i \gg 0$.  A standard change of rings spectral sequence (\cite[10.73]{Rotman}) induces the following long exact sequence  (cf. \cite[(1.4)]{HJ}),

\[\begin{tikzpicture}[descr/.style={fill=white,inner sep=1.5pt}]
        \matrix (m) [
            matrix of math nodes,
            row sep=1em,
            column sep=1.8em,
            text height=1.5ex, text depth=0.25ex
        ]
        { M \otimes_{R/\underline{x}R} N & \Tor^1_{R}(M,N) & \Tor^1_{R/\underline{x}R}(M,N) & 0 \\
           \Tor^1_{R/\underline{x}R}(M,N) & \Tor^2_{R}(M,N) & \Tor^2_{R/\underline{x}R}(M,N) & \mbox{} \\
          \Tor^2_{R/\underline{x}R}(M,N)  & \Tor^3_{R}(M,N) & \Tor^3_{R/\underline{x}R}(M,N) & \mbox{} \\
          \mbox{}  & \mbox{} & \hspace{1.4cm} \cdots \hspace{.25cm} & \mbox{} \\
        };

        \path[overlay,->, font=\scriptsize,>=latex]
        (m-1-1) edge (m-1-2)
        (m-1-2) edge (m-1-3)
        (m-1-3) edge (m-1-4);
        \path[overlay,->, font=\scriptsize,>=latex]
        (m-2-3) edge[out=365,in=185] (m-1-1)
        (m-2-2) edge (m-2-3)
        (m-2-1) edge (m-2-2)
        (m-3-3) edge[out=365,in=185] (m-2-1);
        \path[overlay,->, font=\scriptsize,>=latex]
        (m-3-2) edge (m-3-3)
        (m-3-1) edge (m-3-2);
        \path[overlay,->, font=\scriptsize,>=latex]
        (m-4-3) edge[out=365,in=185] (m-3-1);   
\end{tikzpicture}\]
Since $\Tor^{R/\underline{x}R}_i(M,N)=0$ for $i \gg 0$, it follows that $\Tor^R_i(M,N)=0$ for $i \gg 0$. By assumption  $M$ or $N$ must have finite projective dimension over $R$. Since $\underline{x} \in \m\setminus\m^2$, the same holds over $R/\underline{x}R$ (\cite[Corollary 27.5]{Nagata}). The conclusion follows.  

We now show $(1)$ implies $(3)$.  Let $\underline{x}$ be  a maximal $R$-regular sequence  in $\m\setminus \m^2$.  As $(1)$ implies $(2)$, we know $R/\underline{x}R$ satisfies trivial $\Tor$-vanishing. Since $R/\underline{x}R$ is Artinian, it is complete, and also $R/\underline{x}R \cong \hat{R}/\underline{x}\hat{R}$.  Then $\hat{R}$ satisfies trivial $\Tor$-vanishing by Proposition \ref{finite flat} (1).   

\end{proof}

We are now ready to prove Theorem \ref{equiv}.

\begin{proof}[Proof of Theorem \ref{equiv}]
We begin  with (2). By Remark \ref{repl} and by  passing to sufficiently high syzygies, we may assume $M$ and $N$ are MCM. The result now follows from Lemma \ref{exp} (2) and (3).

We continue with (1). Assume $R$ satisfies trivial $\Tor$-vanishing.  By Proposition \ref{pass}, $\hat{R}$ satisfies trivial $\Tor$-vanishing.  Since $\hat{R}$ has a canonical module, $\hat{R}$ has trivial $\Ext$-vanishing as well, and then so does $R$ by Proposition \ref{finite flat} (2).
\end{proof}

\begin{remark}
In general, trivial vanishing need not ascend along flat local maps.  Indeed, any complete equicharacteristic CM local ring $(S,\fn)$ is a finite flat extension of a regular local ring  $R$ (Noether normalization lemma). But there exist such  $S$ that do not satisfy trivial vanishing (e.g., $S$ is a complete complete intersection with $\codim(S) \gs 2$). 

Trivial vanishing is also not preserved by modding out a regular element in $\m^2$; see Theorem \ref{necess} (2).
\end{remark}

We now summarize results in the literature that  characterize the trivial vanishing condition in  small  codimension.  We recall that  $R$ has an {\it embedded deformation} if there exists  a a local ring $(Q,\fp)$ such that $R\cong Q/(\underline{f})$, for some regular sequence $\underline{f}\subseteq \fp^2$ .\

\begin{thm}\label{necess}
Let $R$ be a CM local ring and set $c=\codim(R)$, then following hold. 
\begin{enumerate}
\item[$(1)$] If $c\ls 1$, $R$ satisfies trivial vanishing. $($\cite[1.9]{HW}$)$ 
\item[$(2)$]  If $c\gs 2$ and $\hat{R}$ has an embedded deformation, then $R$ does not satisfy trivial vanishing. $ ($\cite[4.2]{Sega}$)$ 
\item[$(3)$]  If $c=2$,  then $R$ satisfies trivial vanishing if and only if it is not a complete intersection if and only if it is a Golod ring. $ ($\cite{Sch}, \cite[1.1]{JS}$)$ 
\item[$(4)$]  If  $c = 3$, or $c = 4$ and $R$ is Gorenstein,  then $R$ satisfies trivial vanishing if and only if $\hat{R}$ has no embedded deformation. $ ($\textnormal{Proposition \ref{codim3}}, \cite[2.1]{Sega}$)$ 
\item[$(5)$] There exists  $R$ with $c=4$, and another Gorenstein $R$ with $c=5$, such that $\hat{R}$ has no embedded deformation and $R$ does not satisfy trivial vanishing.  $($\cite[3.10]{GP90}, \cite[3.3]{JS}$)$.
\end{enumerate}
\end{thm}

\section{Main results}\label{mainResults}

This section includes our  main results. We   present several sufficient conditions for a CM local ring to satisfy trivial vanishing. 

The following is an important lemma for the proofs of our results; it  allows us to relate the vanishing of Tor  
with the growth of the Betti numbers
of the modules involved.

\begin{lemma}\label{fkCom}
Let $R$ be a CM local ring and set $c =  \codim(R)$, $e=e(R)$, and $\ell=\elll(R)$. Let  $M$ and $N$ be non-free MCM $R$-modules   and  assume $\Tor_i^R(M,N)=0$ for $i\gg 0$, then

\begin{enumerate}
\item[$(1)$] $$  (\lr_R(M)+1)(\lr_R(N)+1)\ls e.$$

\item[$(2)$] If the limits  in $\lr_R(M)$ and $\lr_R(N)$  exist, we also have 
$$ \lr_R(M)\lr_R(N)\ls e-c-\ell+2.$$
\end{enumerate}

\end{lemma}

\begin{proof} 
 We begin with (1). By assumption,  $\Tor^i_R(\Omega_j^R (M),N)=0$ for $i\gg 0$ and $j\gs 0$. Therefore,   from Lemma \ref{superThm} it follows that
\begin{equation}\label{Finq}
\lr_R(N)\ls \uind_R(\Omega_j^R (M)) -1 \text{ for every } j\gs 0. 
\end{equation}
Fix $n\in \NN_{n>0}$ and pick $j\in \NN$ such that $\beta_{j+1}^R(M) > (\lr_R(M)-1/n) \beta_j^R(M)$.  From the additivity of multiplicities, we have $$e_R(\Omega_j^R(M))+e_R(\Omega_{j+1}^R(M))=e_R(R^{\beta_j^R(M)})=e\beta_j^R(M) .$$ Therefore, either $$ e_R(\Omega_j^R(M))\ls  \dfrac{e\beta_j^R(M) }{\lr_R(M)-1/n+1},$$ or, $$e_R(\Omega_{j+1}^R(M))\ls    \dfrac{ (\lr_R(M)-1/n)e\beta_j^R(M) }{\lr_R(M)-1/n+1}<\dfrac{e\beta_{j+1}^R(M) }{\lr_R(M)-1/n+1}.$$
Hence, $\min\{\uind_R(\Omega_j^R (M)), \uind_R(\Omega_{j+1}^R (M))\}\ls  \dfrac{e}{\lr_R(M)-1/n+1}$. Then from \eqref{Finq} we have
$e\gs (\lr_R(M)-1/n+1)(\lr_R(N)+1)$. The result of (a) now follows by taking $\displaystyle\lim_{n\to\infty}$.

We continue with  (2). We may assume $R$ is Artinian (see Remark \ref{reductions}). Let $i\in \NN$ and consider the module $\Gamma = \Omega_{i+1}^R (M) \otimes_R \Omega_{i+1}^R (N)$. By the vanishing of Tor assumption we have $\Gamma\hookrightarrow \fm^2 (R^{\beta_{i}^R(M)}\otimes_R R^{\beta_{i}^R(N)})$. Therefore, from \cite[2.1]{GP90}, it follows that 
$$(e-c-\ell+2)\beta_{i}^R(M)\beta_{i}^R(N)\gs \mu(\Gamma)=\beta_{i+1}^R(M)\beta_{i+1}^R(N).$$
The conclusion now follows by dividing by  $\beta_{i}^R(M)\beta_{i}^R(N)$ and taking $\displaystyle\lim_{i\to\infty}$.
\end{proof}

We now present the  first main result of this section which gives a sufficient numerical condition for a ring to satisfy trivial vanishing.

\begin{thm}\label{main}
Let $R$ be a CM local ring and set $c =  \codim(R)$, $e=e(R)$, and $\ell=\elll(R)$. If  
\begin{equation}\label{firstIn}
e<\frac{4c+2\ell-1-\sqrt{8c+4\ell-3}}{2},
\end{equation} 
then $R$ satisfies trivial  vanishing. 
\end{thm}

\begin{proof} 
By Theorem \ref{equiv} it suffices to show $R$ satisfies trivial Tor-vanishing. In order to simplify the notation, set $b:=2c+\ell-1$, then we have $$2e<2b+1-\sqrt{4b+1},$$ therefore $b-e>0$. We proceed by contradiction. Assume there exists  $M$ and $N$ of infinite projective dimension with $\Tor^R_i(M,N)=0$ for $i\gg 0$. By  replacing $M$ and $N$ with sufficiently high syzygies, we may assume $M$ and $N$ are MCM.  From Proposition \ref{eventInc} (2) it follows that $\min\{\lr_R(M),\,\lr_R(N)\}\gs b-e$. Therefore, by  Lemma \ref{fkCom} (1) we have $e\gs (b-e)^2$. After rearranging terms, this inequality becomes $$e^2-(2b+1)e+b^2\ls 0,$$
which implies $e\gs \frac{2b+1-\sqrt{4b+1}}{2}$. This contradiction proves the result.
 \end{proof}

We are now ready to present the second main result of the section. In this theorem, we prove a stronger result than Theorem \ref{main} under the  extra assumption that the ring is  generalized Golod (cf. \cite[3.1]{HS04}).

\begin{thm}\label{maingengolod}
Let $R$ be a CM local ring that is also generalized Golod and set $c =  \codim(R)$, $e=e(R)$,  and $\ell=\elll(R)$. Assume 
\begin{equation}\label{secondIn}
e\ls2c+\ell-3
\end{equation}
then for any $R$-modules $M$ and $N$, we have
\begin{enumerate}
\item[$(1)$] If $\Tor_i^R(M,N)=0$ for $i\gg 0$, then $\cx_R(M)\ls 1$, or, $\cx_R(N)\ls 1$.
\item[$(2)$] If $\Ext^i_R(M,N)=0$ for $i\gg 0$, then $\cx_R(M)\ls 1$, or, $\px_R(N)\ls 1$.
\item[$(3)$] If, additionally, $e\ls2c+\ell-4$, then $R$ satisfies trivial  vanishing.
\end{enumerate}
\end{thm}
\begin{proof}
By Proposition \ref{pass} we may assume $R$ is complete. By Theorem  \ref{necess} (1), (3) and \cite{AvrBuch} the conclusion holds if $c\ls 2$.  Therefore, we can assume $c\gs 3$. By  the assumption and Proposition \ref{eventInc} (2), (3), every $R$-module has either infinite complexity, or complexity at most one.  We note that (2) follows from (1), Lemma \ref{exp} (2), and Remark \ref{sameCx}. Then we only show the proof of $(1)$ and $(3)$.

We begin with $(1)$.  By  replacing $M$ and $N$ with sufficiently high syzygies, we may assume $M$ and $N$ are MCM.   We proceed by contradiction. Assume $\cx_R (M)=\cx_R(N)=\infty$. 
 In order to simplify the notation, set $b:=e-c-\ell+2$. By Proposition \ref{ratLimExists}  the limits in $\lr_R(M)$ and $\lr_R (N)$ exist and are larger than one. Let $\gamma$ be either one of these limits. By dividing the inequality in Proposition \ref{eventInc} (1) by $\beta^R_{n-1}($-$)$ and taking $\displaystyle\lim_{n\to\infty}$ we obtain  $\gamma\gs c -\frac{b}{\gamma}$. 
Consider the polynomial $p(z)=z^2 - cz +b $ and notice $p(\gamma)\gs 0$. By assumption $p(z)$ has  positive discriminant,
$c^2-4b \gs 4(2c+\ell-3-e)> 0.$
Hence, $p(z)$ has only real roots. Moreover, $c^2-4b  -(c-2)^2\gs 0.
$
Therefore, $\frac{c-\sqrt{c^2-4b }}{2}\ls 1.$
 Since $\gamma>1$,  we have $\gamma \gs \frac{c+\sqrt{c^2-4b }}{2}$. Therefore, 
 $$ \lr_R(M) \lr_R(N)\gs \frac{(c+\sqrt{c^2-4b })^2}{4} > \frac{(c+\sqrt{c^2-4b })(c-\sqrt{c^2-4b })}{4} =b,$$
 which contradicts Lemma \ref{fkCom}. This finishes the proof of (1).

We continue with (3).  Under the  extra assumption $e\ls 2c+\ell-4$, Proposition \ref{eventInc} (2) implies that if an $R$-module has finite complexity, then  it has finite projective dimension. This concludes the proof.
\end{proof}

\begin{remark}\label{proofOUL}
We remark that the proof of Theorem \ref{maingengolod} requires less from $R$ than being as in Example \ref{gGolod}, or even generalized Golod. Indeed, it only requires that the limit  $\displaystyle\lim_{n\rightarrow \infty}\frac{\beta_{n+1}^R(M)}{\beta_{n}^R(M)}$ exist under vanishing of $\Tor$ hypothesis. This strategy was used in \cite[3.1]{HS04} to prove a similar result under the assumption $\fm^3=0$. 
\end{remark}

\begin{remark}
By Theorem \ref{necess} (2), it follows that there is no CM local ring $R$ for which $\hat{R}$ has an embedded deformation, $\codim(R)\gs 2$, and such that it satisfies  \eqref{firstIn} or the assumption of Theorem \ref{maingengolod} (3).
\end{remark}

A CM local ring $R$ is {\it stretched} if $e(R) = \codim(R)+\elll(R)-1$ (\cite{Sally}). In \cite[5.4]{CDGetc} it is proved that this class of rings are generalized Golod. As a corollary of Theorem \ref{maingengolod} we recover the following  result which originally appeared in  \cite{Gupta}.

\begin{cor}\label{stretched}
Let $R$ be a stretched CM local ring such that $\codim(R)\gs 3$, then $R$ satisfies trivial vanishing.
\end{cor}

The following proposition provides information from vanishing of Ext and Tor over a generalized Golod rings when one knows the Poincar\'e series and multiplicity (cf. \cite[1.5]{Sega}). We recall the definition of $\DenRt$ was given in Theorem \ref{gGolod}.

\begin{prop}
 
Let $R$ be a CM local ring that is also generalized Golod and and set $c =  \codim(R)$, $e=e(R)$,  and $\ell=\elll(R)$. Let $\rho = \min\{\sqrt{e}-1,\,\sqrt{e-c-\ell+2} \}$ and assume $\DenRt$ does not have real roots in the interval $[1/\rho, 1)$, then
\begin{enumerate}
\item[$(1)$] If $\Tor_i^R(M,N)=0$ for $i\gg 0$, then either $\cx_R(M) $ or $\cx_R(N)$ is finite.
\item[$(2)$] 
If $\Ext^i_R(M,N)=0$ for $i\gg 0$, then either $\cx_R(M)$ or $\px_R(N)$ is finite.
\end{enumerate}
\end{prop} 
\begin{proof}
 As in the proof of Theorem \ref{maingengolod}, it suffices to show (1), and also we may assume $M$ and $N$ are MCM. We proceed by contradiction. Assume  $\cx_R (M)=\cx_R (N)=\infty$.  By \cite[(2.3)]{Avr89}, Theorem \ref{AvrDen}, and \cite[Corollary]{Sun98} it follows that  $\min\{\lr_R(M),\lr_R(N)\}>\rho$. This  contradicts Lemma \ref{fkCom}. 
\end{proof}

We obtain the following result about trivial vanishing for rings of small multiplicity. 
We remark that this result is optimal in the sense that  the conclusion does not hold for higher multiplicities (see Example \ref{TheE}). (cf. \cite[1.2]{GP90})

\begin{prop}\label{lowdim}
Let $R$ be a CM local ring such that $\codim(R)\neq 1$. If $e(R)\ls 7$, or $e(R)\ls 11$ and $R$ is Gorenstein, then $R$ satisfies trivial vanishing if and only if $\hat{R}$ has no embedded deformation.
\end{prop}

We first discuss when  $\hat{R}$ has embedded deformations.
\begin{remark}\label{lowmultRmk}
 Write $\hat{R}=P/I$ where $(P,\fn)$ is regular and $I\subseteq \fp^2$. Set $n=\mu(I)$. Under the assumptions of Proposition \ref{lowdim}, one observes from the following proof that if $R$ does not satisfy trivial vanishing, then $\codim(R)\ls 3$, or $\codim(R)=4$ in the Gorenstein case.  In these cases, by \cite[3.6]{Avr89} we have that $\hat{R}$ has a embedded deformation if  and only if it is  a complete intersection, of type {\bf H($\bm{n-1,n-2}$)}  for $n\gs 4$ (cf. Remark \ref{deform3Rmk}), or of type {\bf G($\bm{n-1}$)}  for $n\gs 5$
\end{remark}

\begin{proof}[Proof of Proposition \ref{lowdim}]
 The forward direction is  Theorem \ref{necess} (2); we prove the backward one.   
  By  Theorem \ref{equiv} it suffices to show $R$ satisfies trivial Tor-vanishing.
Set $c=\codim(R)$ and $\ell=\elll(R)$. By Theorem \ref{necess} (3),(4) we may further assume $c\gs 4$, and $c\gs 5$ if $R$ is Gorenstein. Moreover, by \cite[1.1]{JS} we may also assume $\ell\gs 3$.

If $e(R)\ls 7$, the result follows by applying  Theorem \ref{main}, then we assume $R$ is Gorenstein and $e(R)\ls 11$. By \cite[6.9]{Gupta} the ring $R$ is generalized Golod and then  Theorem \ref{maingengolod} implies that if $\Tor_i^R(M,N)=0$ for $i\gg 0$, then either $M$ or $N$ have  complexity at most one (notice $e(R)=c+2$ if $\ell=3$). Hence, by \cite[1.1]{GP90}, one of these modules has finite projective dimension, finishing the proof.
\end{proof}

The following is a general result about codimension three CM local rings.   A similar result  for Gorenstein local rings up to codimension four was shown by \c{S}ega in \cite{Sega}. Moreover,  the authors of the  ongoing work \cite{AINS2} have informed us that they show trivial vanishing  holds if $\mu(\fm)-\depth(R)\ls 3$ (without the CM assumption) in all but some exceptional cases.

We say that a finitely generated $R$-module  $M$ is {\it periodic of period $p$ after $n$ steps} if there exist $ p\in \NN$ such that $\Omega^R_i(M)\cong \Omega^R_{i+p}(M)$ for every $i\gs n$ (cf. \cite{GP90}).

\begin{prop}\label{codim3}
Let $R$ be a CM local ring of dimension $d$  and  $\codim(R) = 3$. Then $R$ satisfies trivial vanishing if and only if $\hat{R}$ has no embedded deformation.

 Moreover, if $R$ is not a complete intersection and $M$ and $N$ are $R$-modules such that $\Tor_i^R(M,N)=0$ for $i\gg 0$, then either $M$ or $N$ is  periodic of period two after at most $d+1$ steps.
\end{prop}

We need the following observation prior to the proof the result
\begin{remark}\label{deform3Rmk}
In \cite[3.6]{Avr89}, Avramov explains the structure of  CM rings of $\codim(R)\ls 3$ and such that $\hat{R}$  has embedded deformation. Write $\hat{R}=P/I$ where $(P,\fn)$ is regular and $I\subseteq \fp^2$. Indeed, $\hat{R}$ has a embedded deformation if  and only if it is either a complete intersection, or of type {\bf H($\bm{n-1,n-2}$)} where $n=\mu(I)\gs 4$. In the latter case, $I$ can be assumed to be of the form $J+(x)$, where $J$ is an ideal of height two that is generated by the maximal minors of a $(n-1)\times (n-2)$  matrix with entries in $\fn$, and $x$ is regular on $P/J$.
\end{remark}

\begin{proof}[Proof of Proposition \ref{codim3}]
 We begin with the first statement. First we note that the forward implication follows from Theorem \ref{necess} (2), hence we  may focus on the backward implication.  By Proposition \ref{pass} we may assume $R$ is complete. 
 By Cohen  Structure Theorem there exits a regular local ring $(P,\fn,k)$ such that $R\cong P/I$ with $I\subseteq  \fn^2$ and by  \cite[1.3]{BE} the minimal $P$-resolution of $R$ has a DG-algebra structure. Set $T$ to be the graded $k$-algebra $\Tor^P(R,k)$. In what follows, for a graded $k$-algebra $B$ such that $B_0=k$ and $B_i=0$ for $i\gs 4$, we denote by $\H_B$ the vector $(\dim_k B_1,\,\dim_k B_2,\,\dim_k B_3)$.

Set $n=\mu(I)$, $\tau=r(R)$ the  type of $R$, and note that by assumption we have $n\gs 4$. Then we have a minimal free $S$-resolution 
$$0\to S^\tau\to S^{n+\tau-1}\to S^n\to S\to R\to 0.$$ Equivalently, $\H_T =(n,\,n+\tau-1,\, \tau)$. By \cite[2.1]{AKM}, there exists a graded $k$-algebra $A$ and a vector space $W$ such that $T$ is isomorphic to the trivial extension $A\ltimes W$. We have the following possibilities for $\H_A$ (\cite[2.1]{AKM})
\

\begin{itemize}
\item[] $R$ is of type {\bf TE}, then $\H_A=(3, \ 3, \ 0)$.
 
 \item[] $R$ is of type {\bf B}, then $\H_A=(2, \ 3, \ 1)$.
 
\item[] $R$ is of type {\bf G($\bm{r}$)}, then $\H_A=(r, \ r, \ 1)$ for some $r\gs 2$.

 \item[] $R$ is of type {\bf H($\bm{p,q}$)} then $\H_A=(p+1, \ p+q, \ q)$ for some $p,\,q\gs 1$.
 
\end{itemize}

 By comparing $\H_T$ and $\H_A$ on each of the cases above, we can immediately see that $W\neq 0$ unless $R$ is Gorenstein of type {\bf G($\bm{n}$)}, or $R$ is of type {\bf H($\bm{n-1,\tau}$)}. If $R$ is not as in the latter exceptional cases, then it follows from \cite[5.3]{AINS} that $R$ satisfies trivial vanishing. If $R$ is of type {\bf G($\bm{n}$)} then $R$ also satisfies trivial vanishing by \cite[2.3]{Sega}.  It remains to consider the case {\bf H($\bm{n-1,\tau}$)}.   By \cite[1.1]{COW} we must have $p-1=q=\tau=n-2$. The first statement now follows by Remark \ref{deform3Rmk}.

We continue with the second statement.  By  replacing $M$ and $N$ with sufficiently high syzygies, we may assume $M$ and $N$ are non-free MCM and $R$ has an embedded deformation. Let $(Q,\fp)$ be a  local ring such that $R\cong Q/(\underline{f})$ where $\underline{f}\subseteq \fp^2$ is a regular sequence.  Since $R$ is not a complete intersection, we must have $\codim(Q)= 2$, and  then $Q$ is a Golod ring \cite{Sch}. Moreover,  $\Tor_i^Q(M,N)=0$ for $i\gg 0$ (\cite[2.6]{AY}).  By Theorem \ref{necess} (3) and the Auslander-Buchsbaum formula, it follows that either $M$ or $N$ has  projective dimension one over $Q$.  The conclusion now follows from \cite[(1.6)(II)]{Avr89}. 
\end{proof}

\subsection{Examples} In the remaining part of this section, we provide examples that discuss the sharpness of our results and the necessity of the conditions. As with previous results, we set $c =  \codim(R)$, $e=e(R)$,  and $\ell=\elll(R)$.

We begin with the following example.

\begin{example}
Let $R = k[[ x,y ]]/(x^2,y^2)$. We note that in this case $e = 4$, $c=2$, and $\ell=3$. Therefore, $R$ satisfies equality in \eqref{firstIn}. However, $R$ does not satisfy trivial vanishing by Theorem \ref{necess} (2). In fact, we may consider  $M = R/xR$ and $N=R/yR$ for which we have $\Tor_i^R(M,N)=0$ for every $i>0$.
\end{example}

\begin{example}\label{SW}
Let $R_1 = k[[ x,y ]]/(x,y)^2$, $R_2=k[[z]]/(z^2)$, and $R=R_1\otimes_k R_2$. We note that $R$ has an embedded deformation, then this is an example of the exceptional case in Proposition \ref{codim3}.  By Theorem \ref{necess} (2), $R$ does not satisfy trivial vanishing. If fact,  we may consider $M = R/(x,y)R$ and $N=R/zR$. We also remark that, since $e=6$, $c=3$, and $\ell=3$, the inequality in \eqref{secondIn} is satisfied, but not the extra condition of Theorem \ref{maingengolod} (3). Hence, this example shows that this extra condition is necessary to guarantee trivial vanishing.
\end{example}

\begin{example}\label{TheE}
Using  \cite[3.1 3.4]{GP90}, the authors in \cite[3.3]{JS} provided some examples of rings that do not satisfy trivial vanishing. The first one is a Gorenstein ring such that  $e=12$, $c=5$, and $\ell=4$. The second one has $e=8$, $c=4$, and $\ell=3$. The completion of each these rings does not have an embedded deformation (\cite[3.10]{GP90}). In these examples the right hand side of \eqref{firstIn} are  approximately $9.9$ and approximately $7.3$, respectively.
\end{example}

\begin{example}\label{Gup}
Let $R$ be a Artinian Gorenstein local ring such that $c=6$, $\ell=4$, and  $\mu(\fm^2)=4$. Therefore, $e=12$. 
By \cite[6.9]{Gupta} the ring $R$   is generalized Golod. Thus, by Theorem \ref{maingengolod} (3), $R$  satisfies trivial vanishing. We note that this conclusion cannot  be obtained from Theorem \ref{main}.
\end{example}

We now present examples that show that the converse of our theorems do not hold.

\begin{example}
 Assume $R$ is a stretched CM local  ring such that $c\gs 3$. By Corollary \ref{stretched} $R$ satisfies trivial vanishing. However, if one fixes $e$ and $c$, then for $\ell\gg 0$, the  inequality  \eqref{firstIn} does not hold. If $R$ is Artinian Gorenstein and {\it almost stretched}, i.e., $\mu(\fm^2)\ls 2$, then $R$ is generalized Golod and satisfies trivial vanishing \cite[6.3]{Gupta}. However if $R$ is Artinian and has $h$-vector $(1,5,2,2,2,2,1)$, then  \eqref{secondIn} fails.
\end{example}

\begin{example}
For every integer $\tau\gs 2$, Yoshino constructed non-Gorenstein rings that do not satisfy trivial vanishing and such that $e=2\tau+2$, $c=\tau+1$, and $\ell=3$ (\cite[4.2]{Yoshino}). These examples show that for rings that do not satisfy trivial vanishing, the distance between $e$ and the right hand side of \eqref{firstIn} can be as large as possible, even if we assume the ring is not Gorenstein. Moreover, for $\tau=2$, this example also shows that the extra condition of Theorem \ref{maingengolod} (3) is necessary to conclude trivial vanishing.
\end{example}

\section{The uniform Aulander condition and the Auslander-Reiten conjecture}\label{ARSection}

In this section we prove that the Auslander-Reiten conjecture holds for a ring with small multiplicity with respect to its codimension. We also provide new classes of rings that satisfy the following condition.

\begin{definition}\label{UAC}
The ring $R$  satisfies the {\it uniform Auslander condition $($UAC$)$} if there exists $n\in \NN$ such that whenever $\Ext^i_R(M,N)=0$ for $i\gg 0$, this vanishing holds for $i\gs n$. 
\end{definition}

  In \cite{CH} the authors prove that if $R$ satisfies the UAC then it satisfies Conjecture \ref{ARC}. 
Some classes of rings that satisfy this condition are complete intesection rings (\cite{AvrBuch}), Golod rings (\cite{JS}),  CM local  rings of codimension at most two, and Gorenstein rings of codimension at most four (\cite{Sega}). 
 We refer the reader to \cite[Appendix A]{CH} for a survey on the topic and to \cite{AINS2} for related results.
 
 As a corollary of our results in the previous section, we are able to provide new classes of rings that satisfy UAC and hence Conjecture \ref{ARC}. In particular,  part (3)  implies that the examples constructed in \cite{JS} for the failure of UAC are minimal with respect to codimension, and part (4) shows they are minimal with respect to multiplicity.

\begin{cor}\label{finiteVirtual}
Let $R$ be a CM local ring and assume $R$ satisfies one of the following conditions.
\begin{enumerate}
\item[$(1)$] $R$ satisfies the inequality \eqref{firstIn}.
\item[$(2)$] $R$ is as one of $(3)$-$(8)$ in Example \ref{gGolod} and satisfies the inequality \eqref{secondIn}, or $R$ is any generalized Golod ring and satisfies the assumption of Theorem \ref{maingengolod} $(3)$.
\item[$(3)$] $\codim(R)\ls 3$.
\item[$(4)$] $e(R)\ls 7$, or $e(R)\ls 11$ and $R$ is Gorenstein.
\end{enumerate}
Then $R$ satisfies UAC and thus the Auslander-Reiten conjecture $($Conjecture \ref{ARC}$)$ holds for $R$.
\end{cor}
\begin{proof}
We can assume that $R$ is complete (\cite[5.5]{CH}).  Let $M$ and $N$ be $R$-modules such that $\Ext^i_R(M,N)=0$ for every $i\gg 0$. By Remark \ref{repl} and by replacing $M$ for a sufficiently high syzygy, we may assume $M$ and $N$ are MCM. 

If $R$ is as in (1), then it satisfies trivial vanishing and hence UAC.

 Assume now  $R$ is as in (2). By Theorem \ref{maingengolod},  $M$ or $N^\vee$ has finite complexity. Therefore, by
 \begin{itemize}
 \item[$\vartriangleright$] \cite[1.5]{Avr89} for (3)-(6),
 \item[$\vartriangleright$]  \cite[proof of 5.2]{KP} for (7), and
 \item[$\vartriangleright$]  \cite[2.4]{Avr89}, \cite[5.2]{K} for (8),
\end{itemize}  
 we have that $M$ or $N^\vee$ has finite {\it complete intersection $($CI$)$ dimension } (see \cite{AvrBuch}). Therefore, $R$ satisfies UAC by \cite[4.7, 4.1.5]{AvrBuch}.  The second statement of (2) is clear as $R$ satisfies trivial vanishing.

We now consider (3), we may assume $\codim(R)=3$ (\cite[1.1]{JS}). In this  case Proposition \ref{codim3} and its proof show that when $R$ does not satisfy trivial vanishing, then one of $M$ or $N^{\vee}$ has finite CI-dimension. Thus the conclusion follows as before.

For (4), from the proof of Proposition \ref{lowdim} it follows that $R$ satisfies trivial vanishing if $\codim(R)\gs 4$, or if $\codim(R)\gs 5$ in the Gorenstein case.  The result now follows from (3) and \cite{Sega}. 
\end{proof}

The following is the main theorem of this section.

\begin{thm}\label{mainAR}
Let $R$ be a CM local ring and set $c=\codim(R)$. Assume $R$ satisfies one of the following conditions.
\begin{enumerate}
\item[$(1)$] $e(R)\ls \frac{7}{4}c+1$.  
\item[$(2)$]   $e(R)\ls  c+6$ and $R$ is Gorenstein.
\end{enumerate}
Let $M$ be an $R$-module such that $\Ext_i^R(M,M\oplus R)=0$ for $i\gg 0$.  Then $M$ has finite projective dimension, i.e.,  the Auslander-Reiten conjecture $($Conjecture \ref{ARC}$)$  holds for $R$.
\end{thm}
\begin{proof}
Proceeding as in Proposition \ref{finite flat} (2), we may assume $R$ is complete and $k$ is infinite. We may also replace $M$ by $\Omega_j^R(M)$ for any  $j\in \NN$  (\cite[1.2]{HL}) to assume $M$ is MCM and $\Ext_i^R(M,M\oplus R)=0$ for $i> 0$. Hence by standard arguments we may  assume $R$ is Artinian (see Remark \ref{reductions}).

We prove (1) first.  Set $\ell = \elll(R)$ and $b = 2c+\ell -1$. Notice that the right hand side of \eqref{firstIn} is  $f(b):=b+\frac12 - \sqrt{b+\frac14}$, which is  an increasing function for $b\gs 0$.  By \cite[4.1]{HS04} the conclusion holds if $\ell \ls 3$.  Thus we may assume $\ell\gs 4$ and hence $f(b)\gs 2c+\frac72 -\sqrt{2c+\frac{13}{4}}.$
Let $g(c)$ be the right hand of the last inequality, viewed as a function of $c$. It suffices to show $g(c)> \frac{7}{4}c+1$ which is equivalent to 
$\frac{1}{16}(c-6)^2+\frac{3}{4} >  0$. The result follows. 

We now prove (2). If $e(R)\ls 11$, the conclusion follows from Corollary \ref{finiteVirtual} (4), and if $c\gs 7$, it follows from part (1) of this theorem. Hence  we may assume  $c=6$ and $e(R)=12$. By  \cite[4.1]{HS04} and Theorem \ref{main},  it remains  to consider when $\ell = 4$. Therefore we are in the situation of  Example \ref{Gup}, and then $R$ satisfies trivial vanishing. This finishes the proof.
\end{proof}

\begin{cor}\label{e8}
Let $R$ be a CM local ring such that $e(R)\ls 8$. Then the Auslander-Reiten conjecture holds for $R$.
\end{cor}
\begin{proof}
 By Corollary \ref{finiteVirtual}  we may assume $\codim(R)\gs 4$. The conclusion then follows from Theorem \ref{mainAR}.
\end{proof}

\begin{remark}
We note that if $e(R)=9$, the only CM local rings for which we do not know Conjecture \ref{ARC} holds are those  whose Artinian reductions have $h$-vectors $(1,4,3,1)$, $(1,4,2,2)$, or $(1,4,2,1,1)$.  If $e(R)=12$ and $R$ is Gorenstein, then we do not know if this conjecture holds when these  $h$-vectors are $(1,5,5,1)$, $(1,5,4,1,1)$, or $(1,5,3,2,1)$; the case $(1,5,2,2,1,1)$ follows from Theorem \ref{maingengolod} or from \cite[6.5]{Gupta}.
\end{remark}

\section{Criteria for the Gorenstein property}\label{Gorenst}

In this section we expand upon several results in the literaure which provide criteria for $R$ to be Gorenstein based on vanishing of Ext.

We begin with the following lemma. By using this result with $M=\omega_R$ , we drop the generically Gorenstein condition from \cite[2.4]{JL08} (cf. \cite[2.4]{JL07}). This in turn  allows us to provide a more general statement for \cite[2.4]{JL08}. 

\begin{lemma}[{ \cite[A.1]{CSV}}]\label{firstBetti}
Let $R$ be a CM local ring  and $M$ be an $R$-module such that $e_R(M)=e(R)$.   If $M$ is not free, then $\beta_1^R(M) \gs \beta_0^R(M)$.
\end{lemma}

We obtain the following criterion for the Gorenstein property. Similar results have appeared in the work of several authors (see \cite{HH}, \cite{JL08}, \cite{Ulrich}). We recall that the notation  $\uind_R(M)$ was introduced in  Definition \ref{uind}.

\begin{thm}\label{testGor}
Let $R$ be a CM local ring of dimension $d$ and  with canonical module $\w_R$. Let  $M$ be a CM $R$-module with $\uind_R(M)<2$. Assume $\Ext_R^{i}(M,R)=0$ for every $d-\dim M+1\ls i\ls d+1 $.  Then $R$ is Gorenstein.  
\end{thm}

\begin{proof}
The result follows from Lemma \ref{firstBetti} by applying Lemma \ref{superThm} (2) with $N=\w_R$.
\end{proof}

In \cite{JL07}, Jorgensen and Leuschke ask the following question.

\begin{question}[{\cite[2.6]{JL07}}]\label{jorleu}
If $R$ is a CM  local ring with canonical module $\w_R$, does $\beta_1(\w_R) \ls \beta_0(\w_R)$ imply that $R$ is Gorenstein?
\end{question}

They remark that a positive answer to this question would provide improvements to their results \cite[2.2]{JL07} and \cite[2.4]{JL07}.  We give a positive answer to Question \ref{jorleu}  in a particular case, which is  sufficient to produce the desired improvement of \cite[2.4]{JL07}, as well as  \cite[3.4]{HH}  .

We first recall the following result of Asashiba and Hoshino. In the following statement, we denote by $M^*$ the {\it $R$-dual} $\Hom_R(M,R)$.

\begin{lemma}[{\cite[2.1]{AH94}}]\label{twogen}
Let $M$ and $N$ be $R$-modules. Assume $M$ is faithful and that  we have an exact sequence $0 \rightarrow N \xrightarrow{\varphi} R^2 \xrightarrow{\psi} M \rightarrow 0$.  Then there exists  maps $\alpha$, $\beta$, and an isomorphism $\theta$ that make the following diagram commute.
\[\begin{tikzcd}
0 \arrow[r] & N \arrow[r, "\varphi"] \arrow[d, "\alpha"'] & R^2 \arrow[r, "\psi"] \arrow[d, "\theta"'] & M \arrow[r] \arrow[d, "\beta"'] & 0 \\
0 \arrow[r] & M^* \arrow[r, "\psi^*"] & (R^2)^* \arrow[r, "\varphi^*"] & N^* & 
\end{tikzcd}\]
\end{lemma}

We say that an $R$-module $M$ has {\it constant rank} if there is an $r \in \mathbb{N}$ such that $M_{\p} \cong R^r_{\p}$ for all $\p \in \Ass(R)$.  In this case, we refer to $r$ as the \textit{rank} of $M$ and denote it by $\rank M$.

We derive our result on Question \ref{jorleu} from the following more general result.

\begin{prop}\label{totref}
Let $R$ be a  CM local ring. Assume there exists a non-free  faithful MCM  $R$-module  $M$,   such that it has constant rank and $\max\{\beta_0^R(M),\beta_1^R(M)\} \ls 2$.  Then 
\begin{enumerate}
\item[$(1)$] 
$M$ is periodic of period two and is reflexive, i.e., the natural map $M \to M^{**}$ is an isomorphism.
\item[$(2)$] If $R$ is not Gorenstein, then it does not satisfy trivial vanishing.

\end{enumerate}    
\end{prop}

\begin{proof}
We begin with (1). Since $M$ is not free,   $M$ has infinite projective dimension. Therefore, we must have $\beta_0^R(M)=\beta_1^R(M)=2$ and  $\rank(M)=\rank(\Omega_1^R(M))=\rank(\Omega_2^R(M))=1$ by additivity of ranks. By Lemma \ref{twogen}, we have the following commutative diagram, where $\theta$ is an isomorphism.
\[\begin{tikzcd}
0 \arrow[r] & \Omega^R_1(M) \arrow[r, "\varphi"] \arrow[d, "\alpha"'] & R^2 \arrow[r, "\psi"] \arrow[d, "\theta"'] & M \arrow[r] \arrow[d, "\beta"'] & 0 \\
0 \arrow[r] & M^* \arrow[r, "\psi^*"] & (R^2)^* \arrow[r, "\varphi^*"] & (\Omega_1^R(M))^* & 
\end{tikzcd}\]
By commutivity of the diagram, we have that $\alpha$ is injective and, by the snake lemma, we have that $\ker \beta \cong \coker \alpha$.  Since $M$ is MCM and  $\ker \beta$ embeds into $M$, we have that either $\dim (\coker \alpha)=d$, or $\coker \alpha=0$.  
On the other hand, $$\rank(\coker \alpha)=\rank(M^*)-\rank(\Omega_1^R(M))=0.$$ Therefore, $\coker \alpha=0$  and then  $\alpha$ is an isomorphism.  Hence, $\beta$ is an isomorphism as well since 
$$\rank(\coker \beta)=\rank(M)-\rank((\Omega_1^R(M))^*)=0.$$ 
 We conclude that $M\cong M^{**}$, $\Ext^1_R(M^*,R)=0$, and that we have the following exact sequence $0 \to M^* \to R^2 \to M \to 0$.

Since $ (\Omega_1^R(M))^*\cong M$, we have that $(\Omega^1_R(M))^*$ is faithful, and then so is $\Omega_1^R(M)$.  Then we may apply Lemma \ref{twogen} as above to  obtain isomorphisms  $$\Omega_1^R(M)\cong (\Omega_2^R(M))^*\cong (\Omega_1^R(M))^{**}.$$  
We also conclude $\Ext^1_R(M,R)=\Ext^1_R((\Omega_1^R(M))^{*},R)=0$ and that   we have an  exact sequence $0 \to M \to R^2 \to M^* \to 0$. Splicing this sequence with the previous sequence $0 \to M^* \to R^2 \to M \to 0$ repeatedly, we  construct a resolution of $M$. Hence, conclusion of  (1) follows.   

We observe that this argument shows $\Ext_R^i(M,R)=0$ for every $i>0$ and hence  (2) follows. This finishes the proof.
\end{proof}

\begin{remark}
The resolution constructed in the proof of Proposition \ref{totref} is an example of a {\it totally acyclic resolution of $M$}.  That is, a complex $F_\bullet:\cdots \to F_{2}\to F_{1} \to F_0 \to F_{-1} \to F_{-2} \to \cdots$ such that $M \cong \coker(F_1 \to F_0)$ and such that $F$ and $F^*$ are exact. Modules admitting such a resolution are called {\it totally reflexive} \cite[2.4]{AM}.
\end{remark}

With Proposition \ref{totref} in hand, we provide our result on Question \ref{jorleu}. The ring $R$ is {\it generically Gorenstein} if $R_\fp$ is Gorenstein for every $\fp\in \Ass(R)$, or equivalently, if $\omega_R$ has constant rank.

 We recall that  $r(R)$ denotes the type of $R$. 

\begin{cor}\label{nogorc1}
Let  $R$ be a  generically Gorenstein CM local ring such that $r(R)=2$. Then $\beta^R_1(\w_R)>2$. 
\end{cor}

\begin{proof}
We proceed by contradiction. Assume $\beta^R_1(\w_R) \ls 2$, then Proposition \ref{totref} (1) implies that $(\w_R)^*$ is MCM and  we have an  exact sequence $0 \to \w_R \to R^2 \to \w_R^* \to 0$.  By canonical duality, we have the sequence  $0 \to (\w_R^*)^{\vee} \to \w_R^2 \to R \to 0$ which splits. Then, by applying canonical duality again, we conclude  $\w_R$ free. The latter implies $R$ is Gorenstein, which contradicts the assumption  $r(R)=2$.
\end{proof}

The following is the main result of this section. This theorem provides  direct improvements to \cite[2.4]{JL07} and \cite[3.4]{HH}. 

\begin{thm}\label{stronghh}
Let $R$ be a generically Gorenstein CM  local ring of dimension $d$ and with canonical module $\w_R$.  Assume there exists an MCM $R$-module $M$  with  $\uind_R(M) \ls 2$ and such that $\Ext^i_R(M,R)=0$ for $1 \ls i \ls d+1$. Then $R$ is Gorenstein.
\end{thm}

\begin{proof}
From Lemma \ref{exp} (1) it follows that $M \otimes_R \w_R$ is MCM. By Lemma \ref{exp} (2) we have an exact sequence 
\begin{equation}\label{sfr}
0\to M\otimes_R \Omega^1_R(\w_R)\to M\otimes_R R^{r(R)}\to   M\otimes_R \w_R \to 0
\end{equation} 
whence it follows $M\otimes_R \Omega^1_R(\w_R)$ is MCM as well. Therefore, 
\[e_R(M)=e_R(M \otimes_R \w_R) \gs \mu(M \otimes_R \w_R)=\mu(M)r(R).\]
This implies $r(R) \ls 2$. If $r(R)=1$, $R$ is Gorenstein.  Then we may assume $r(R)=2$ which implies $\Omega^1_R(\w)$ has rank $1$ by \eqref{sfr}. Therefore,
\[e_R(M)=e(M \otimes_R \Omega^1_R(\w_R)) \gs \mu(M \otimes_R \Omega^1_R(\w_R))=\mu(M)\mu(\Omega^1_R(\w_R)).\]
Hence $\mu(\Omega^1_R(\w_R)) \ls 2$ which contradicts  Corollary \ref{nogorc1}. This finishes the proof.
\end{proof}

\section*{Acknowledgments}
We are grateful to Hailong Dao,  
Sean Sather-Wagstaff, and Ryo Takahashi for helpful discussions. We especially thank Sean Sather-Wagstaff for showing us some of the details of Example \ref{SW}.

\end{document}